\documentclass[11pt]{article}

\usepackage{a4wide}
\usepackage{graphicx}
\usepackage{latexsym}
\usepackage{epsfig}
\usepackage{amssymb}
\usepackage{amstext}
\usepackage{amsgen}
\usepackage{amsxtra}
\usepackage{amsgen}
\usepackage{amsthm}
\usepackage{color}
\usepackage{booktabs}
\usepackage{pdfsync}

\usepackage{subcaption} 

\usepackage{chngcntr}

\def\RR{\mathbb{R}}

\newtheorem{thm}{Theorem}[section]

\newtheorem{rem}[thm]{Remark}

\theoremstyle{definition}
\newtheorem{lemma}[thm]{Lemma}

\numberwithin{equation}{section}

\newcommand{\R}{\mathbb{R}}

\title{Parameter identification in a structured population model}
\author{Alexander Lorz\thanks{Sorbonne Universit\'e, CNRS, Laboratoire Jacques-Louis Lions, F-75005 Paris, France
(on leave). {alexander.lorz.uni@gmail.com}} \and Jan-Frederik Pietschmann\thanks{Technische Universität Chemnitz, Fakult\"at f\"ur Mathematik, Reichenhainer Str. 41, 09126 Chemnitz Germany. jfpietschmann@math.tu-chemnitz.de} \and Matthias Schlottbom\thanks{Department of Applied Mathematics, University of Twente,
P.O. Box 217, 7500 AE Enschede, The Netherlands.
{m.schlottbom@utwente.nl}}}

\begin{document}

\maketitle
{\em Preprint: \today}
\begin{abstract}
 We study parameter identification problems in a structured population model without mutations. Given measurements of the total population size or critical points of the population, we aim to recover its growth rate, death rate or initial distribution. We present uniqueness results under suitable assumptions and present counterexamples when these assumptions are violated. Our results a supplemented by numerical studies, either based on Tikhonov regularization or the use of explicit reconstruction formulas.
\end{abstract}

\section{Introduction}
This paper is concerned with the theoretical and numerical study of several inverse problems in structured population models. These models describe the coevolution of a population where individuals have a distinct quantitative trait, such as their size. The evolution of the number of individuals with given trait $n=n(t,x)$ is assumed to be governed by two effects: Interaction among individuals and interaction with their environment. In general, interactions between individuals are due to competition (e.g. for a common food source) or by random mutations. Here we consider the case where an individuals' offspring has the same trait as its parents, thus neglecting the effect of mutations. This leads to a model of the form
\begin{align}\label{eq:model}
 \partial_t n(t,x) &= s[n]n,\quad x \in \R,\, t \in [0,T],\\
 \label{eq:init}
 n(0,x) &= n_0(x).
\end{align}
The selection rate (or selective pressure) $s[n]$ introduces coupling with respect to the $x$ variable.

The dynamics of such equations has been studied extensively by many authors, see, e.g., \cite{Desvillettes2008,JG11,LP14}. Besides existence and uniqueness of solutions, their long time behavior is analyzed. Depending on the particular form of $s$ it is expected that only a few traits survive for large times, i.e., that the solution converges to a finite sum of Dirac measures. We refer to \cite{Desvillettes2008,Lorz2011,Lorz2015} for more details. This is strongly related to the notion of evolutionary stable strategy (ESS) and we refer the reader to \cite{Maynard1973}. Also note that similar models can also be derived from a stochastic models with finite populations, cf. \cite{Champagnat2006,ChampagnatStochastic2008,DieckmannStochastic1996}.

The dynamics of \eqref{eq:model}--\eqref{eq:init} are determined by the structure of $s[n]$, and knowledge of $s[n]$ allows for prediction of the evolution of the population at future times.
In this work we are interested in identifying the model parameter $s[n]$ from observational data of the solution to \eqref{eq:model}--\eqref{eq:init} in the class of logistic type selection rates, i.e.,
\begin{align}\label{eq:selection}
 s[n] = p(x) - d(x)\rho(t).
\end{align}
Here the parameters to be identified are the reproduction rate $p$ and the trait-dependent weight function $d$ of the death rate $d\rho$,
where 
\begin{align}\label{eq:defrho}
 \rho(t) &= \int n(t,x)\;dx
\end{align}
denotes the total mass of the population at time $t$. Selection rates of form \eqref{eq:selection} are frequently used in the literature, see for example \cite{Roughgarden1979theory,Barabas2009}, yet sometimes with $\rho$ defined as a weighted integral over $n$. In our case, since $\rho(t)$ is simply the total mass, all individuals are in competition with one another, independent of their particular trait.

Typical data that we consider consist of the total population size $\rho(t)$, $0\leq t\leq T$, or of tuples $(\bar x,t)$ of the location of critical values $\bar x$ of $n(t,\cdot)$.
More precisely, we address the following inversion problems:
\begin{enumerate}
 \item[(P1)] Given measurements of $\rho(t)$ on $[0,T]$, determine either the function $p(x)$, $d(x)$ or $n_0(x)$.
 \item[(P2)] Given measurements of critical points of $n(t,\cdot)$, $t\in [0,T]$, determine either $p(x)$, $d(x)$ or $n_0(x)$.
\end{enumerate}
As will be elaborated below, there exist a number of transformations that can be applied to the parameters $p$ and $d$ yet leave the quantities $\rho$ and / or the critical points of $n$ unchanged. In these situations one cannot expect any positive identification result which is directly reflected in the assumption we have to make in our uniqueness theorems. More precisely, for (P1), we are able to give a positive identification result under suitable monotonicity assumptions on the parameters and present explicit counterexamples when these assumptions are violated. In situations when uniqueness is guaranteed, we present numerical reconstructions using Tikhonov regularization, and we verify convergence under a standard smoothness assumption. 
For (P2), we derive explicit formulas for the derivatives $p'$, $d'$ and $n_0'$, which imply uniqueness and stability with respect to perturbation of the measured data. The latter is demonstrated by numerical examples. Finally, we also comment on the simultaneous identification problem
\begin{enumerate}
 \item[(P3)] Given measurements of $\rho(t)$ as well as the position of critical points, determine both $p(x)$ and $n_0(x)$.
\end{enumerate}
In this case we cannot give a definite answer which is mainly due to the fact that it seems very delicate to combine the nonlocal information contained in $\rho(t)$ with the knowledge of critical points that is purely local. Finally, note that our setup is quite different from more common parameter identification problems for partial differential equations, see e.g. \cite{Isakov06}, since we neither have a differential operator acting in space nor measurements on the boundary.
\smallskip\\
This paper is organized as follows: In Section \ref{sec:properties}, we study the population model and show existence and uniqueness of solutions. In Section \ref{sec:inverse_population} we address (P1), give counterexamples to the identification problem for general parameters, and give classes of parameter functions for which the inverse problems in (P1) can be solved uniquely. In Section~\ref{sec:inverse_critical}, we consider (P2) and present reconstruction formulas for the derivates of the parameter functions evaluated at critical points of the population density, which is followed by a discussion regarding (P3). We present extensive numerical results for the actual reconstruction of the unknown parameters, including different ways to treat the (nonlinear) problem as well as convergence rates in Section \ref{sec:numerics}. Finally, in Section~\ref{sec:outlook}, we give an outlook for a population model with mutation.
\section{Existence of solutions}\label{sec:properties}
Equations \eqref{eq:model}--\eqref{eq:defrho} can be understood as a system of ordinary differential equations (for every point $x\in \R$) coupled via $\rho(t)$, which motivates to rewrite the solution using the following implicit representation
\begin{align}\label{eq:explicit}
 n(t,x) = n_0(x)e^{tp(x) - d(x)\int_0^t \rho(s)\;ds}.
\end{align}
Integrating expression \eqref{eq:explicit} with respect to space yields the following nonlinear fixed-point equation for the total population 
\begin{align}\label{eq:rho_fix}
	\rho(t)= \int_\RR n_0(x)e^{tp(x) - d(x)\int_0^t \rho(s)\;ds}dx,
\end{align}
which is an ordinary differential equation for $R(t)=\int_0^t\rho(s)ds$ with initial data $R(0)=0$. For convenience of the reader and for later reference, we provide a proof of uniqueness and existence of solutions to \eqref{eq:model}--\eqref{eq:defrho}. Let us refer also to \cite[Thm 2.1]{Desvillettes2008} for a similar strategy, yet in different function spaces.
\begin{thm}\label{thm:existence}
	Let $p,d \in L^\infty(\RR)$ be non-negative and let $n_0\in L^1(\RR)$ be non-negative. Then there exists a unique $n\in C^{\infty}([0,T],L^1(\RR))$ and $\rho\in C^\infty([0,T])$ solution to \eqref{eq:model}--\eqref{eq:init}.
\end{thm}

\begin{proof}
The proof relies on Banach's fixed point theorem. For 
$$
M=\{\rho\in L^\infty(0,T):\rho\geq 0\}
$$
define the map $\Lambda:M\to M$ as
\begin{align}\label{eq:def_lambda}
	(\Lambda(\rho))(t)=\int_\R n_0(x) e^{tp(x)-d(x)\int_0^t\rho(s)ds}dx.
\end{align}
By construction, fixed points of $\Lambda$ are solutions to \eqref{eq:rho_fix}. We endow the space $L^\infty(0,T)$ with the norm
\begin{align*}
	\|\rho\|_{\infty,a}=\sup_{0<t<T} |\rho(t)|e^{-at}
\end{align*}
and chose $a=2\|n_0\|_{L^1} \|d\|_\infty e^{T\|p\|_\infty}$. We have $a=0$ when either $n_0\equiv 0$ or $d\equiv0$, and the assertion holds trivially. Let now $a>0$. Obviously, $\Lambda$ is a self-mapping. In order to show that $\Lambda$ is a contraction, we observe that
\begin{align*}
	|e^{-d z}-e^{-dz_0}|\leq d |z-z_0|
\end{align*}
for all $z_0,z\geq 0$. Hence, we obtain for $\rho_1,\rho_2\in M$
\begin{align*}
	|\Lambda(\rho_1)-\Lambda(\rho_2)|(t)&\leq \int_\R n_0(x) e^{tp(x)} |e^{-d(x)\int_0^t\rho_1(s)ds}-e^{-d(x)\int_0^t\rho_2(s)ds}|dx\\
	&\leq \|n_0\|_{L^1} e^{T\|p\|_\infty} \|d\|_\infty\int_0^t|\rho_1(s)-\rho_2(s)|ds\\
	&\leq\|n_0\|_{L^1} e^{T\|p\|_\infty} \|d\|_\infty\|\rho_1-\rho_2\|_{\infty,a} \frac{e^{at}}{a}.
\end{align*}
By the choice of $a$, we thus obtain 
\begin{align*}
	\|\Lambda(\rho_1)-\Lambda(\rho_2)\|_{\infty,a}\leq \frac{1}{2}\|\rho_1-\rho_2\|_{\infty,a},
\end{align*}
which shows that $\Lambda$ is a contraction. Banach's fixed point theorem implies the existence and uniqueness of $\rho\in M$ such that $\rho=\Lambda(\rho)$. Defining $n(t,x)$ via \eqref{eq:explicit} yields the unique solution to \eqref{eq:model}--\eqref{eq:init}. In addition, since $t\mapsto \int_0^t \rho(s)ds \in W^{1,\infty}(0,T)$, we infer that $n(t,x)\in W^{1,\infty}(0,T)$ a.e. $x$. The regularity assumptions on $p$, $d$ and $n_0$ yield that $n\in W^{1,\infty}(0,T;L^1(\RR))$. Using \eqref{eq:defrho}, we then obtain $\rho\in W^{1,\infty}(0,T)$. Repeating these arguments, we obtain higher order differentiability in time of $\rho$ and $n$.
\end{proof}
%
%
\section{Identification from knowledge of the total population size}\label{sec:inverse_population}
In the following we address inverse problem (P1).
In general, the coefficient $p$ is not uniquely determined given measurements of the total population $\rho$ as shown by the following examples.
\begin{itemize}
 \item[(i)] Translational invariance: Let $n_0(x) = 1$ for $x\in \R$, $d=0$ and let $c>0$ be arbitrary. In addition, choose a compactly supported function $p(x)$ and define the function $\bar p(x) := p(x+c)$. Solving \eqref{eq:model}--\eqref{eq:init} with parameters $p$ and $\bar p$, respectively, yields the same function $\rho(t)$. 
 \item[(ii)] Symmetry: Let $d(x)=d(-x)$, $n_0(x)=n_0(-x)$, and let $p_1(x)$ be arbitrary. If we define $p_2(x)=p_1(-x)$, then $n_2(x,t)=n_1(-x,t)$, and  $\rho_1(t)=\rho_2(t)$ for $t\geq 0$.
\end{itemize}
These examples suggest to consider the class of strictly monotone coefficient functions $p$.
\begin{thm}\label{thm:ident_p}
 Let $n_0\in C^0(\RR)$ be nonnegative with compact and connected support. Assume that $d(x)=d>0$ is constant.
 Denote by $p_1$ and $p_2$ continuous and strictly monotone functions on the support of $n_0$ such that $p_1'p_2'>0$, and let $n_1$ and $n_2$ denote the solutions to \eqref{eq:model}--\eqref{eq:init} with $p$ replaced by $p_1$ and $p_2$, respectively. Then, with $\rho_1$ and $\rho_2$ being the respective population sizes we have
\begin{align*}
 \rho_1 = \rho_2\text{ on } [0,T]\text{ implies } p_1 = p_2\text{ on } \mathrm{supp}(n_0).
\end{align*}
\end{thm}
\begin{proof}
 By assumption $\rho=\rho_1=\rho_2$, and it follows from \eqref{eq:defrho} that
\begin{align}\label{eq:identity}
 \int_{\RR} n_0 e^{tp_1} e^{-d\int_0^t \rho(s)ds}\;dx = \int_\RR n_0 e^{tp_2} e^{-d\int_0^t \rho(s)ds}\;dx.
\end{align} 
Since $d$ is constant, this implies
\begin{align*}
 \int_{\RR} n_0 e^{tp_1} \;dx = \int_\RR n_0 e^{tp_2} \;dx.
\end{align*}
Using monotonicity of $p_1$ and $p_2$, we can transform each of the integrals, using either $y=p_1(x)$ or $y=p_2(x)$ as new variables, respectively, to obtain
\begin{align*}
 \int_{\RR} \left( \frac{n_0(p^{-1}_1(y))}{p_1'(p_1^{-1}(y))}\chi_{\mathcal{P}_1}(y) - \frac{n_0(p^{-1}_2(y))}{p_2'(p_2^{-1}(y))}\chi_{\mathcal{P}_2}(y)\right) e^{ty}\;dy = 0,
\end{align*}
where we also used that $p_1'p_2'>0$. Here, $\mathcal{P}_i=p_i(\mathcal{S})$ for $\mathcal{S}={\rm supp}(n_0)$, $i=1,2$, and $\chi_{\mathcal{P}_i}$ denotes the indicator function of the set $\mathcal{P}_i$. Since $\mathcal{S}$ is a compact interval and $p_i \in C^0(\mathcal{S})$, $\mathcal{P}_i$ are compact intervals.
Differentiation with respect to $t$ and evaluating the result for $t=0$ then yields, for every $k\ge 0$,
\begin{align}\label{eq:N0dx}
 \int_{\mathcal{P}_1\cup \mathcal{P}_2} \left( \frac{n_0(p^{-1}_1(y))}{p_1'(p_1^{-1}(y))}\chi_{\mathcal{P}_1}(y) - \frac{n_0(p^{-1}_2(y))}{p_2'(p_2^{-1}(y))}\chi_{\mathcal{P}_2}(y)\right) y^k\;dx = 0.
\end{align}
The term in brackets is continuous as a function of $y$ due to the construction of $\mathcal{P}_i$, $i=1,2$.
Since $\mathcal{P}_1\cup \mathcal{P}_2$ is compact, a density argument yields
\begin{align}\label{eq:diff_N}
 \frac{n_0(p^{-1}_1(y))}{p_1'(p_1^{-1}(y))}\chi_{\mathcal{P}_1}(y)  - \frac{n_0(p^{-1}_2(y))}{p_2'(p_2^{-1}(y))}\chi_{\mathcal{P}_2}(y) =0
\end{align}
for all $y\in \mathcal{P}_1\cup \mathcal{P}_2$. This readily implies ${\mathcal{P}_1}\setminus {\mathcal{P}_2}=\emptyset$ and ${\mathcal{P}_2}\setminus {\mathcal{P}_1}=\emptyset$, and hence $\mathcal{P}_1\cup \mathcal{P}_2=\mathcal{P}_1\cap \mathcal{P}_2$, i.e., $\mathcal{P}_1= \mathcal{P}_2$. 
Introducing the primitive of $n_0$, i.e., 
\begin{align*}
	N_0(x)=\int_{x_0}^x n_0(z) dz,
\end{align*}
where $x_0 = \min\mathcal{S}$, we see that \eqref{eq:diff_N} is equivalent to
\begin{align*}
  \frac{d}{dy}\left(N_0(p_1^{-1}(y)) - N_0(p_2^{-1}(y))\right) = 0
\end{align*}
for all $y\in \mathcal{P}:=\mathcal{P}_1=\mathcal{P}_2$. The assumption $p_1'p_2'>0$ then implies $p_1(x_0)=p_2(x_0)$, and hence
$N_0(p_1^{-1}(y)) = N_0(p_2^{-1}(y))$ for all  $y\in\mathcal{P}$.
Using the definition of $N_0$ we thus obtain
\begin{align*}
	\int_{p_2^{-1}(y)}^{p_1^{-1}(y)}n_0(z)dz=0
\end{align*}
for $y\in \mathcal{P}$. Since, $p_i^{-1}(\mathcal{P})=\mathcal{S}$, $i=1,2$, and $n_0$ is positive in the interior of $\mathcal{S}$, we deduce that $p_2^{-1}(y)=p_1^{-1}(y)$ for all $y\in\mathcal{P}$, i.e., $p_1(x)=p_2(x)$ for all $x\in\mathcal{S}$.
\end{proof}
\begin{rem}[Identification of $d$ and $n_0$]
Interchanging the roles of $d$ and $p$ in the above examples shows that, in general, uniqueness of $d$ cannot be expected from knowledge of $\rho$ only.
With similar arguments as in the proof of Theorem~\ref{thm:ident_p}, one can, however, prove uniqueness of $d$ in the class of strictly monotone functions (either increasing or decreasing) given measurements of $\rho(t)$ and knowledge of $n_0$ and constant $p$.
Moreover, one can show that for $p$ and $n_0$ arbitrary, knowledge of $\rho(t)$, $t\geq 0$, uniquely determines constant parameters $d$.
The transformation $y=p(x)$ in the proof of Theorem~\ref{thm:ident_p} can also be used to identify compactly supported initial data $n_0$ if $p$ is strictly monotone and $d$ is constant. We leave the details to the reader.
\end{rem}

\section{Identification in critical points of the population}\label{sec:inverse_critical}
Above we have shown that, under appropriate assumptions, the total population size contains sufficient information for the determination of some of the parameters of the problem. These results, however, do not provide an explicit reconstruction formula. In this section, we show that knowledge of the critical points of the population density can be used to directly compute derivatives of the unknown parameters.\\
Before we state the results, we discuss properties of the critical points of $n$ in some detail.
\subsection{The critical points of $n$}\label{sec:critical}
We call a point $\bar x\in {\rm supp}(n_0)$ critical for $n$ if there exists a $t\geq 0$ such that $\partial_x n(t,\bar x)=0$.
\begin{lemma}\label{lem:formula_diff_log}
  	Denote by $n$ the solution to \eqref{eq:model}--\eqref{eq:init} for  differentiable parameter functions $d$ and $p$.
	Then, any critical point $\bar x$ of $n$ is characterized by
		\begin{align}\label{eq:rec_from_max}
			 (\ln (n_0(x)))'_{\mid x=\bar x} =  d'(\bar x) \int_0^t\rho(s)ds-t p'(\bar x) .
		\end{align}
\end{lemma}
\begin{proof}
	Using the chain rule, we see that $\bar x$ is also a critical point of $\ln n$, i.e.,
	\begin{align*}
		\partial_x (\ln(n(t,x)))_{\mid x=\bar x}=0.
	\end{align*}
	On the other hand, from the solution formula \eqref{eq:explicit}, we deduce that 
	\begin{align*}
		\ln(n(t,x)) = \ln (n_0(x)) + t p(x) - d(x) \int_0^t\rho(s)ds,
	\end{align*}
	so we obtain the result by differentiation with respect to $x$ and evaluation at $x=\bar x$.
\end{proof}
Assuming that $d$ is constant, the critical points of $n$ are, therefore, those $\bar x\in {\rm supp}(n_0)$ for which $t\geq 0$ exists with
\begin{align}\label{eq:cond_crit}
	\frac{n_0'(\bar x)}{n_0(\bar x)}+tp'(\bar x)=0.
\end{align}
We distinguish three cases:
\begin{itemize}
	\item[(i)] For $n_0'(x)p'(x)>0$, the point $x$ is never a critical for $n(t,\cdot)$.
	\item[(ii)] For $n_0'(x)p'(x)<0$, there exists a unique $t=-n_0'(x)/(n_0(x)p'(x))$ for which $x$ is a critical point of $n(t,\cdot)$.
	\item[(iii)] For $n_0'(x)p'(x)=0$, if $p'(x)=0$, then \eqref{eq:cond_crit} implies $n_0'(x)=0$, and $x$ is a critical point of $n(t,\cdot)$ for all $t\geq 0$. Otherwise, if $p'(x)\neq 0$, then $x$ is critical point for $n(t,\cdot)$ only for $t=0$.
\end{itemize}
A similar discussion applies for $p$ constant and $d$ variable; or $n_0$ constant and $p$ and $d$ variable.
\subsection{Identification of a single parameter}
As a direct consequence of Lemma~\ref{lem:formula_diff_log} we obtain the following reconstruction formulas for the derivatives of the parameters.
\begin{thm}\label{thm:recon_p_d_prime}
	  Let $T>0$, and denote by $n$ the solution to \eqref{eq:model}--\eqref{eq:init} for differentiable parameter functions $d$ and $p$ and differentiable initial datum. Furthermore, let $\bar x$ be a critical point of $n(\cdot,t)$ for some $t>0$.
	
	(i) If $d$ is constant, then $p'(\bar x)$ is uniquely determined by $n_0$, i.e., 
		\begin{align}\label{eq:cond_max_p}
			 p'(\bar x)=-\frac{n_0'(\bar x)}{t n_0(\bar x)}.
		\end{align}
	(ii) If $p$ is constant, then $d'(\bar x)$ is uniquely determined by $n_0$ and $\int_0^t\rho(s) ds$, i.e.,
	\begin{align}\label{eq:cond_max_d}
		 d'(\bar x)=\frac{n_0'(\bar x)}{ n_0(\bar x)\int_0^t \rho(s)ds}.
	\end{align}
	(iii) $(\ln(n_0(x)))'_{\mid x=\bar x}$ is uniquely determined by $p'(\bar x)$, $d'(\bar x)$ and $\int_0^t\rho(s)ds$ by \eqref{eq:rec_from_max}.
\end{thm}
\begin{rem}
	It can be easily seen from the solution formula \eqref{eq:explicit} that the functions $n(t,x)$ and $n_c(x,t)$, which are solutions to \eqref{eq:model}--\eqref{eq:defrho} for parameters $(p,d)$ and $(p+c,d)$ with constants $c,d\in\RR$, respectively, share the same critical points. In this sense, the previous theorem cannot be improved without further assumptions. A similar conclusion holds true for parameter pairs $(p,d)$ and $(p,d+c)$.
\end{rem}
\begin{rem}
	 In the situation of Theorem~\ref{thm:recon_p_d_prime},
	if the closure of the set of critical points coincides with the support of $n_0$, then $p$ is determined up to an additive constant. If in addition $\rho(t)$ is known for some $t>0$, then this additive constant is fixed, i.e., $p$ is unique.
\end{rem}
\subsection{Remarks on simultaneous identification}
Simultaneous identification of multiple parameters or their derivatives is difficult.
Counting dimensions, it is to be expected that measurements of the one dimensional function $\rho(t)$ is not sufficient to simultaneously recover two the parameter functions, which is supported by the following examples
\begin{itemize}
	\item[(i)] Let $n_0$ be any compactly supported function with $\int n_0 dx=a>0$. Let $d_1(x)$ and $d_2(x)$ be arbitrary functions, and define $p_i(x)=a d_i(x)$, $i=1,2$.  Then $n_i(x,t)=n_0(x)$ solves \eqref{eq:model}--\eqref{eq:init} with $\rho_i(t)=\rho(0)=a$. Hence, knowledge of $\rho$ does not allow to identify $p$ and $d$ simultaneously.
\item[(ii)] Let $n_0$ be any function supported on $[0,1]$, $d=0$, and let $p_i:[0,1]\to[0,1]$, $i=1,2$, be two invertible functions that satisfy $p_i(0)=0$ and $p_i(1)=1$. We define the initial datum as $n_0^i(x)=n_0(p_i(x))p'_i(x)$, and denote $n_i$ the corresponding solutions to \eqref{eq:model}--\eqref{eq:init}. Using the substitution $y=p_i(x)$, we obtain that
\begin{align*}
	\rho_i(t)&=\int_0^1 n_0^i(x) e^{p_i(x)t}dx=\int_0^1 n_0(y) e^{yt}dy,
\end{align*}
	i.e., $\rho_1(t)=\rho_2(t)$.
	Hence, it is not possible to determine $n_0$ and $p$ from $\rho$. This argument can be extended to $d>0$.
\end{itemize}
In Section~\ref{sec:inverse_critical}, we have seen that measuring the critical points allows for reconstruction of derivatives of one of the parameters.
The discussion in Section~\ref{sec:critical} shows that if $x$ is a critical point of $n$ for two distinct times, say $t_1, t_2\geq 0$, then $n_0'(x)=0$ and $p'(x)=0$ are uniquely determined given that $d\in\RR$ is constant. Similarly, $n_0'(x)=0$ and $d'(x)=0$ if $p\in\RR$. Using \eqref{eq:rec_from_max} this reasoning can be extended to non-constant $p$ and $d$, and to obtain formulas for $d'(x)$ and $p'(x)$ given $n_0'(x)/n_0(x)$ and $\rho(t)$, which is
\begin{align*}
	\begin{pmatrix} \int_0^{t_1}\rho(s)ds & -t_1\\\int_0^{t_2}\rho(s)ds & -t_2\end{pmatrix}\begin{pmatrix} d'(\bar x)\\ p'(\bar x)\end{pmatrix}=\frac{n_0'(\bar x)}{n_0(\bar x)}\begin{pmatrix}1\\1\end{pmatrix}.
\end{align*}
We note that, in general, the matrix in the above linear system might be singular, thereby allowing for multiple solutions or none.
We note that identifying two of the parameter functions from knowledge of $\rho$ and $x(t)$, where $x(t)$ denotes a curve of critical points, with $x'(t)\neq 0$ remains an open problem.
\section{Reconstructions}\label{sec:numerics}

\subsection{Reconstructions from the total population size}
In this section we assume knowledge of the total population size $\{\rho(t):0\leq t\leq T\}$ in order to determine the parameter function $p(x)$.
Theorem~\ref{thm:ident_p} shows that measuring the total population size is sufficient in order to uniquely reconstruct the parameter $p$ as long as $d$ is a constant and $p$ is either strictly increasing or strictly decreasing.
Contrary to the situation of Theorem~\ref{thm:recon_p_d_prime}, there are, however, no explicit reconstruction formulas available.
We thus propose to use a variational regularization technique to numerically reconstruct $p$ from measurements of the (noisy) total population size $\rho^\delta(t)$, where $\delta > 0$ denotes the noise level. 
In the following two subsections we discuss two approaches to define suitable Tikhonov regularizations in Hilbert spaces.

\subsubsection{Fully nonlinear forward operator}
We begin with the obvious definition of the nonlinear forward operator
\begin{align*}
	F: X=H^1(\mathcal{S})\to Y=L^2(0,T),\quad p\mapsto \rho\quad\text{where}\quad \rho=\Lambda_p(\rho).
\end{align*}
Here, the subscript $p$ should emphasize the dependence on $p$ of the map $\Lambda$ as defined in \eqref{eq:def_lambda}. The choice of $X=H^1(\mathcal{S})$ is motivated by the continuity of the embedding $H^1(\mathcal{S})\hookrightarrow L^\infty(\mathcal{S})$, which implies that $F$ is well-defined by Theorem~\ref{thm:existence}. 
Denoting by $p_0\in H^1(\mathcal{S})$ some a-priori knowledge, such as a monotonically increasing function, we construct stable approximations to the exact solution $p^\dagger$, which satisfies $F(p^\dagger)=\rho$, by minimizing the Tikhonov functional
\begin{align}\label{eq:Tik}
	\frac{1}{2} \|F(p)-\rho^\delta\|_{Y}^2 + \frac{\alpha}{2}\|p-p_0\|_{X}^2,
\end{align}
over the space $H^1(\mathcal{S})$. Here and in the following we make the assumption that the data perturbation can be estimated as follows
\begin{align}
	\|\rho-\rho^\delta\|_{L^2(0,T)}\leq \delta.
\end{align}
Standard theory of inverse problems can be used to prove existence of minizimers $p_\alpha^\delta$ and stable dependence on the data as long as $\alpha>0$, see e.g. \cite{EHN96}. 
Widely used algorithms to minimize the Tikhonov functional employ the gradient of $F$.
Without proof (which amounts to a lengthy calculation using \eqref{eq:explicit}), we note that $F$ depends smoothly on $p$ and the Fr\'echet derivative is
\begin{align*}
	F'(p):h\mapsto D\quad\text{where}\quad D(t)=\int_\RR \left[th(x)-d(x)\int_0^t D(s)ds \right] n_0(x)e^{pt-d\int_0^t\rho ds}dx,
\end{align*}
for $p,h\in H^1(\mathcal{S})$.
We observe that the definition of $F'(p)h$ constitutes an ordinary differential equation for $\int_0^t D(s)ds$, which yields the explicit formula
\begin{align*}
	(F'(p)h)(t)=D(t)=\int_\RR h(x) n_0(x)\int_0^t \left(\int_0^s e^{pr}dr\right) e^{-d\int_0^s\rho(r)dr} dsdx.
\end{align*}
Using this formula, it is straightforward to obtain a formula for the adjoint operator $F'(p)^*\psi$, $\psi\in L^2(0,T)$, which is defined as the solution to
\begin{align*}
	-\Delta w + w &= n_0(x) \int_0^T \psi(t) \int_0^t \left(\int_0^s e^{p(x)r}dr\right) e^{-d\int_0^s\rho dr}ds dt\qquad \text{in }\mathcal{S},\\
	\partial_n w&=0\quad\text{on } \partial\mathcal{S}.
\end{align*}
It is easy to verify that for all $h,p\in H^1(\mathcal{S})$ and $\psi\in L^2(0,T)$ 
\begin{align*}
	(F'(p)h,\psi)_{L^2(0,T)} = (h, F'(p)^*\psi)_{H^1(\mathcal{S})}.
\end{align*}
Convergence rates for the error $\|p_\alpha^\delta- p^\dagger\|_{H^1(\mathcal{S})}$ follow from assuming a source condition \cite{EHN96}
\begin{align}\label{eq:source_condition}
	p^\dagger-p_0 = F'(p^\dagger)^*w
\end{align}
with sufficiently small $w\in L^2(0,T)$.
In order to approximate minimizers of the Tikhonov functional, we use the iteratively regularized Gauss-Newton (IRGN) method 
\begin{align*}
	p_{k+1}=p_k +  (F'(p_k)^*F'(p_k)+\alpha_k I)^{-1}\big(F'(p_k)^*(\rho^\delta-F(p_k)+\alpha_k (p_k-p_0)\big),
\end{align*}
where $\alpha_k=\max\{\alpha,1/2^k\}$; see \cite{BakKok04} for a convergence analysis if $\alpha=0$ and $\delta=0$. Let us refer to \cite{ES15} for a discussion on the use of the IRGN method to minimize \eqref{eq:Tik} with $\alpha>0$.

\paragraph{Numerical example}
We illustrate the performance of the IRGN method choosing the example $n_0(x)=\cos(\pi x/2)$, for $x\in\mathcal{S}=(-1,1)$, $p^\dagger(x)=e^x$, and $d(x)=1$. The final time is chosen as $T=1$. We choose a spatial grid with spacing $10^{-3}$ and temporal grid with spacing $10^{-2}$.
The initial guess $p_0$ is chosen such that it satisfies \eqref{eq:source_condition} with $w(t)=e^{-t}$. A reconstruction is shown in Figure~\ref{fig:rec_p_nonlinear} together with the convergence rate of the error $\|p_\alpha^\delta-p^\dagger\|_{H^1(\mathcal{S})}$, which exhibits the  rate $O(\sqrt{\delta})$ that is expected for Tikhonov regularization.
The good convergence behavior of the IRGN method can also be seen in Table~\ref{tab:rec_p_nonlinear}.
\begin{figure}
	\includegraphics[width=.48\textwidth]{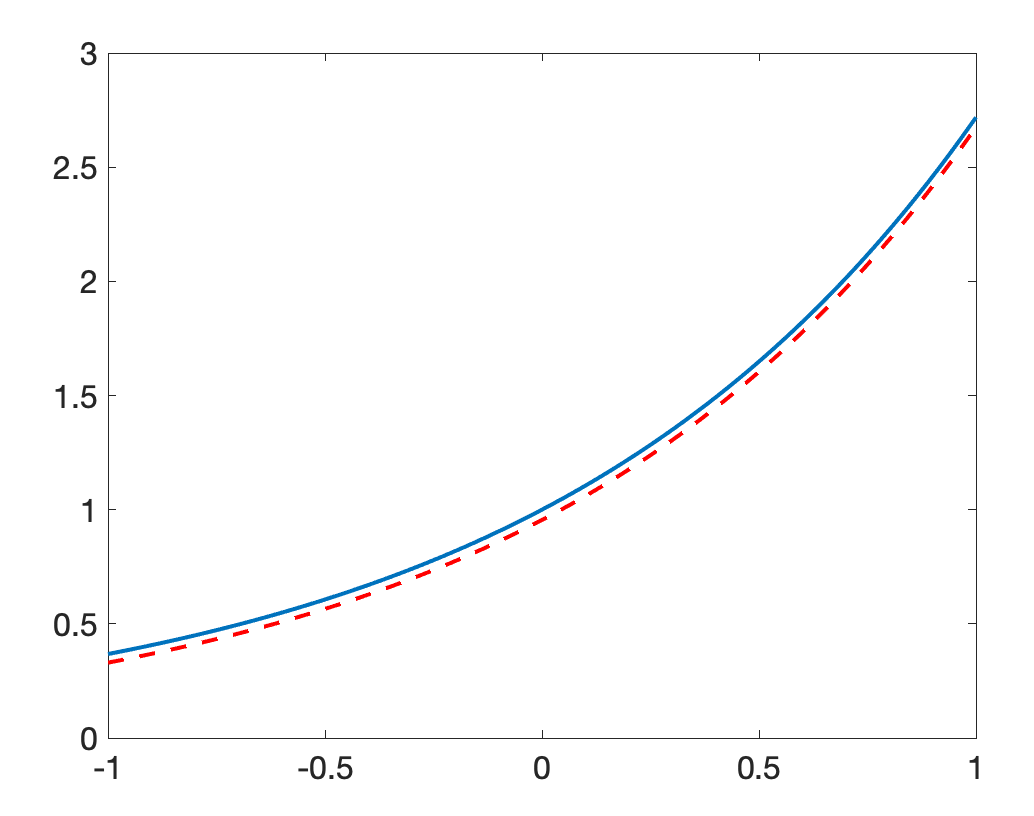}
	\includegraphics[width=.48\textwidth]{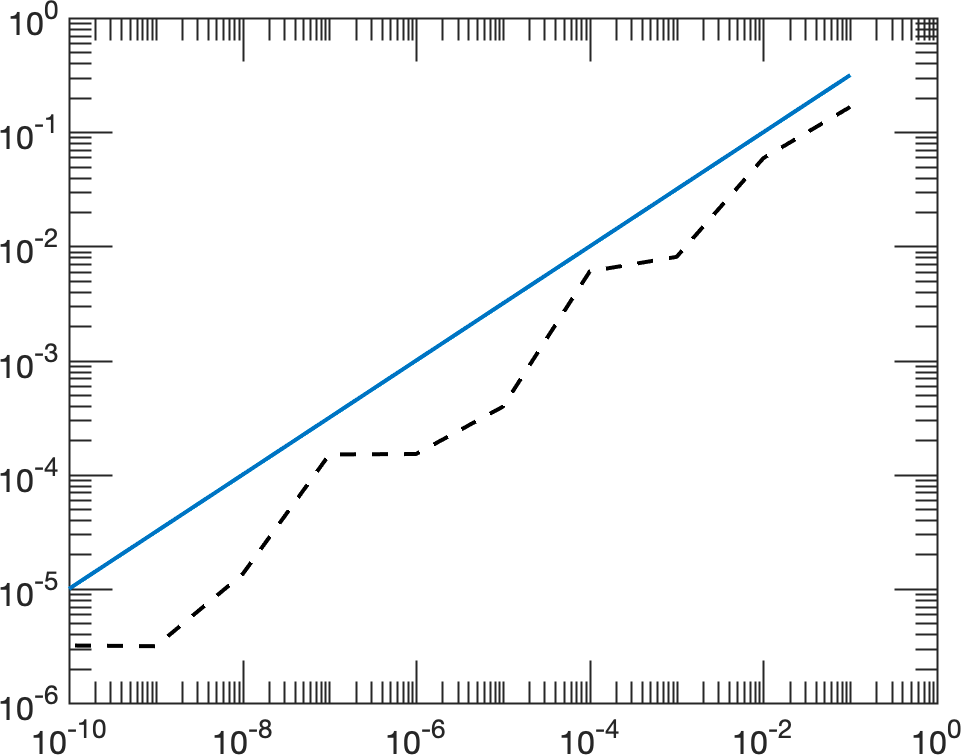}
	\caption{\label{fig:rec_p_nonlinear} Left: $p^\dagger$ (solid line) and corresponding reconstruction $p_\alpha^\delta$ for $\alpha=\delta=1.24\times 10^{-2}$ after $6$ IRGN iterations for minimizing \eqref{eq:Tik}. Right: A plot of the error $\|p_\alpha^\delta-p^\dagger\|_{H^1(\mathcal{S})}$ (dotted) and the curve $\sqrt{\delta}$ (solid) for different values of $\delta$.}
\end{figure}

\begin{table}
\caption{\label{tab:rec_p_nonlinear} Convergence behavior of the IRGN method for the minimization of \eqref{eq:Tik} for different noise levels $\delta$. The error convergence with $O(\sqrt{\delta})$, cf. Figure~\ref{fig:rec_p_nonlinear}.}
\centering
\begin{tabular}{c c c c}
\toprule
$\delta$	& $\|p_\alpha^\delta-p^\dagger\|_{H^1(\mathcal{S})}$ 	& $\|\rho^\delta-F(p_\alpha^\delta)\|_{L^2(0,T)}$	& \# iterations\\
 \midrule
$1.2\times 10^{-1}$		& $1.7\times 10^{-1}$	& $1.4\times 10^{-1}$		& 1	\\
$1.2\times 10^{-2}$		& $5.9\times 10^{-2}$	& $2.4\times 10^{-2}$		& 6\\
$1.2\times 10^{-3}$		& $8.1\times 10^{-3}$	& $2.2\times 10^{-3}$		& 10\\
$1.2\times 10^{-4}$		& $6.0\times 10^{-3}$	& $2.2\times 10^{-4}$		& 14\\
$1.2\times 10^{-5}$		& $3.9\times 10^{-4}$	& $2.1\times 10^{-5}$	& 17\\
$1.2\times 10^{-6}$		& $1.5\times 10^{-4}$	& $1.7\times 10^{-6}$		& 21\\
$1.2\times 10^{-7}$		& $1.5\times 10^{-4}$	& $2.2\times 10^{-7}$ 	&  24\\
$1.2\times 10^{-8}$		& $1.3\times 10^{-5}$	& $2.4\times 10^{-8}$ 	& 27\\ 
$1.2\times 10^{-9}$		& $3.1\times 10^{-6}$ 	& $1.7\times 10^{-9}$& 31\\
$1.2\times 10^{-10}$	& $3.2\times 10^{-6}$ 	& $2.1\times 10^{-10}$ & 34 \\
\bottomrule
\end{tabular}
\end{table}

\subsubsection{Perturbed forward operator}\label{sec:perturbed}
In order to reduce the nonlinearity of the inverse problem, let us present a second choice of forward operator.
Using the data $\rho^\delta$ into the right hand side of \eqref{eq:rho_fix}, we define a perturbed forward operator
\begin{align*}
	F^\delta(p) = \int_\R n_0(x) e^{tp(x)-d(x)\int_0^t\rho(s)^\delta ds}dx.
\end{align*}
 Similar as in the proof of Theorem~\ref{thm:existence} we obtain the following error estimate
\begin{align*}
	\|F^\delta(p)- F(p) \|_{L^2(0,T)} \leq \|n_0\|_{L^1}e^{T\|p\|_\infty} \|d\|_\infty T \|\rho^\delta-\rho\|_{L^2(0,T)}.
\end{align*}
As above, we assume that $n_0$ is compactly supported with support $\mathcal{S}$.
Thus, in view of standard results from the analysis of Tikhonov regularization \cite{EHN96}, we can obtain stable approximations by minimizing the following Tikhonov functional with perturbed forward operator
\begin{align}\label{eq:Tik2}
	\frac{1}{2} \|F^\delta(p)-\rho^\delta\|_{Y}^2 + \frac{\alpha}{2}\|p-p_0\|_{X}^2,
\end{align}
with $Y=L^2(0,T)$ and $X=H^1(\mathcal{S})$, $\mathcal{S}={\rm supp}(n_0)$ and $p_0\in X$.
For completeness, we provide the following result, which is a slight generalization of \cite[Thm 10.3]{EHN96}, see also \cite{Egger2015} for a corresponding result for linear problems.
\begin{lemma}
	Let $F:X\to Y$ be a continuous and weakly lower semi-continuous operator between Hilbert spaces $X$ and $Y$. Let $\delta>0$ and let $F^\delta:X\to Y$ be continuous and weakly lower-semicontinuous such that $\|F^\delta(p)-F(p)\|_Y\leq C(\|p\|_X)\delta$ for all $p\in X$ with a constant $C(\|p\|_{X})$ that depends continuously on $\|p\|_X$. Then, for $\rho,\rho^\delta\in Y$ with $\rho\in R(F)$ and $\|\rho-\rho^\delta\|_Y\leq \delta$, the minimizers $\{p_\alpha^\delta\}$ of \eqref{eq:Tik2} converge along subsequences to a $p_0$-minimum-norm solution of $F(p)=\rho$ with $\delta\to 0$ provided that $\alpha\to 0$ and $\delta^2/\alpha\to 0$. If the $p_0$-minimum-norm solution is unique, then the whole sequence converges to the unique $p_0$-minimum-norm solution
\end{lemma}
\begin{proof}
	The proof is similar to \cite[Thm. 10.3]{EHN96}, and we give only the steps that are different. Let $p^\dagger\in X$ be a $p_0$-minimum-norm solution. Since $\{p_\alpha^\delta\}$ minimize \eqref{eq:Tik2}, we have that
	\begin{align*}
		\frac{1}{2} \|F^\delta(p_\alpha^\delta)-\rho^\delta\|_{Y}^2 + \frac{\alpha}{2}\|p_\alpha^\delta-p_0\|_{X}^2
		&\leq \frac{1}{2} \|F^\delta(p^\dagger)-\rho^\delta\|_{L^2(0,T)}^2 + \frac{\alpha}{2}\|p^\dagger-p_0\|_{X}^2\\
		&\leq 2 C(\|p^\dagger\|_X)^2\delta^2+ \frac{\alpha}{2}\|p^\dagger-p_0\|_{X}^2,
	\end{align*}
	which implies boundedness $\{p_\alpha^\delta\}$ and weak convergence of a subsequence $\{p_{\alpha_k}^{\delta_k}\}$ to $p\in X$. Moreover, we have that
	\begin{align*}
		\|F^\delta(p_\alpha^\delta)-\rho^\delta\|_Y^2\leq 4 C(\|p^\dagger\|_X)^2\delta^2+\alpha\|p^\dagger-p_0\|^2_X.
	\end{align*}
	 By weak lower-semicontinuity of $F$ and using the latter inequality, we obtain that
	\begin{align*}
		\|F(p)-\rho\|_Y&\leq \limsup_k \|F(p_{\alpha_k}^{\delta_k})-\rho^{\delta_k}\|_Y\leq \limsup_k \|F^{\delta_k}(p_{\alpha_k}^{\delta_k})-F(p_{\alpha_k}^{\delta_k})\|_Y + \|F^{\delta_k}(p_{\alpha_k}^{\delta_k})-\rho^{\delta_k}\|_Y\\
		&\leq \limsup_k C(\|p_{\alpha_k}^{\delta_k}\|_X)\delta_k + C\delta_k+C\alpha_k\|p^\dagger-p_0\|_X=0,
	\end{align*}
	where we used continuity of the constant $C(\|p_{\alpha_k}^{\delta_k}\|_X)$ and boundedness of $\{p_{\alpha_k}^{\delta_k}\}$. Thus, $F(p)=\rho$. Proceeding as in the proof of \cite[Thm. 10.3]{EHN96}, we hence obtain the assertion.
\end{proof}

As before, $F^\delta$ is Fr\'echet differentiable with derivative
\begin{align*}
	dF^\delta(p)h = \int_\mathcal{S} h(x) n_0(x) t e^{pt-d\int_0^t\rho^\delta ds} dx,\quad h\in H^1(\mathcal{S}),
\end{align*}
and the adjoint $dF^\delta(p)^*\psi$, $\psi\in L^2(0,T)$, is defined as the solution to
\begin{align*}
	-\Delta w + w &= n_0(x) \int_0^T t\psi(t) e^{pt-d\int_0^t\rho^\delta ds}dt\qquad \text{in }\mathcal{S},\\
	\partial_n w&=0\quad\text{on } \partial\mathcal{S}.
\end{align*}
The Tikhonov functional \eqref{eq:Tik2} can then be minimized as above by the IRGN method, which we consider next.

\paragraph{Numerical Example}
We consider the same example and setup as in the previous section. We observe, that using the perturbed forward operator yields essentially the same results as using the fully nonlinear forward operator. However, the numerical implementation of the perturbed forward operator is simpler. Figure~\ref{fig:rec_p} shows an exemplary reconstruction together with the exact solution and the convergence behaviour of the error $\|p_\alpha^\delta-p^\dagger\|_{H^1(\mathcal{S})}$ for different values of $\delta$. Table~\ref{tab:rec_p} shows, in addition, the convergence of the residuals for different values of $\delta$ and the required IRGN iterations to obtain a suitable reconstruction.
\begin{figure}
	\includegraphics[width=.48\textwidth]{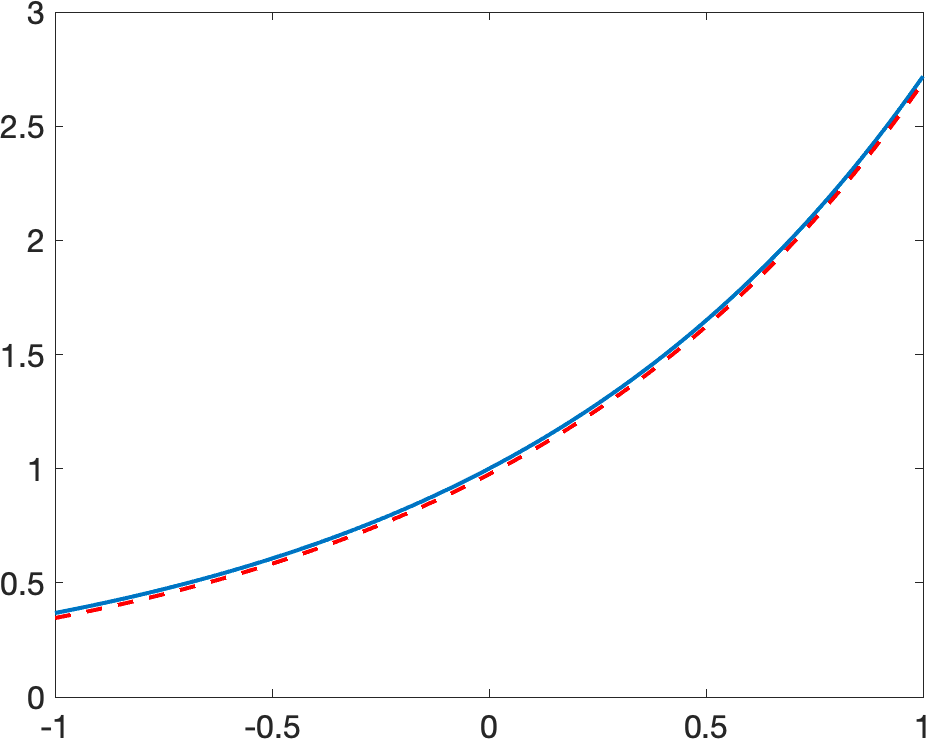}
	\includegraphics[width=.48\textwidth]{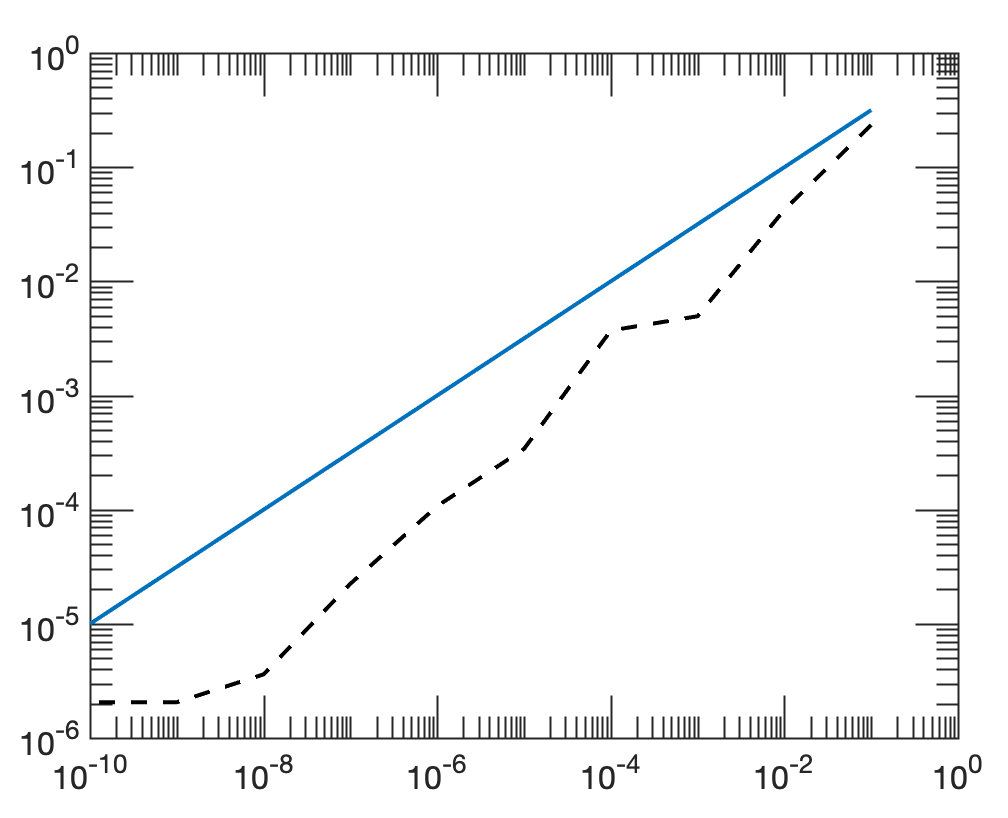}
	\caption{\label{fig:rec_p} Left: $p^\dagger$ (solid line) and corresponding reconstruction $p_\alpha^\delta$ for $\alpha=\delta=1.24\times 10^{-2}$ after $5$ IRGN iterations for minimizing \eqref{eq:Tik2}. Right: A plot of the corresponding errors $\|p_\alpha^\delta-p^\dagger\|_{H^1(\mathcal{S})}$ (dotted) and the curve $\sqrt{\delta}$ (solid) for different values of $\delta$.}
\end{figure}
\begin{table}
\caption{\label{tab:rec_p} Convergence behavior of the IRGN method for the minimization of \eqref{eq:Tik2} for different noise levels $\delta$. The error convergence with $O(\sqrt{\delta})$, cf. Figure~\ref{fig:rec_p}.}
\centering
\begin{tabular}{c c c c}
\toprule
$\delta$	& $\|p_\alpha^\delta-p^\dagger\|_{H^1(\mathcal{S})}$ 	& $\|\rho^\delta-F(p_\alpha^\delta)\|_{L^2(0,T)}$	& \# iterations\\
 \midrule
$2.5\times 10^{-1}$		& $2.3\times 10^{-1}$	& $2.9\times 10^{-1}$		& 1	\\
$2.5\times 10^{-2}$		& $4.2\times 10^{-2}$	& $4.0\times 10^{-2}$		& 5\\
$2.5\times 10^{-3}$		& $4.9\times 10^{-3}$	& $3.4\times 10^{-3}$		& 9\\
$2.5\times 10^{-4}$		& $3.7\times 10^{-3}$	& $4.3\times 10^{-4}$		& 12\\
$2.5\times 10^{-5}$		& $3.4\times 10^{-4}$	& $4.7\times 10^{-5}$	& 15\\
$2.5\times 10^{-6}$		& $1.1\times 10^{-4}$	& $3.4\times 10^{-6}$		& 19\\
$2.5\times 10^{-7}$		& $2.2\times 10^{-5}$	& $3.8\times 10^{-7}$ 	&  22\\
$2.5\times 10^{-8}$		& $3.6\times 10^{-6}$	& $4.4\times 10^{-8}$ 	& 25\\ 
$2.5\times 10^{-9}$		& $2.1\times 10^{-6}$ 	& $3.4\times 10^{-9}$& 29\\
$2.5\times 10^{-10}$	& $2.1\times 10^{-6}$ 	& $4.2\times 10^{-10}$ & 32 \\
\bottomrule
\end{tabular}
\end{table}
%
%

\subsection{Reconstructions using critical points of the population density}
We illustrate the reconstruction formulas given in Theorem~\ref{thm:recon_p_d_prime} by numerical examples. Contrary to Theorem~\ref{thm:ident_p}, Theorem~\ref{thm:recon_p_d_prime} does not require monotonicity of the parameter functions.

\paragraph{Reconstruction of $p'$ from critical points of $n$}
As an initial datum we choose $n_0(x)=\cos(\pi x/2)$, $d(x)=1$ and $p(x)=1+\sin(x)^2$ and we let $x\in (-1,1)$ and $t\in [0,10]$.
For our numerical computations, we discretize $x$ equidistantly with grid spacing $10^{-4}$. Similarly, we discretize time with time step size $10^{-2}$. In our numerical algorithms, given an approximation of $n(t,x)$ we thus compute $\rho(t)$ and $\int_0^t\rho(s)ds$ approximately using quadrature rules. Using these approximations, we compute an approximation of $n$ at the next time instance using \eqref{eq:explicit} with $\int_0^t\rho(s)ds$ replaced by its numerical approximation.

To apply Theorem~\ref{thm:recon_p_d_prime}, we collect the minima and maxima of the approximate population density over time as our data $\{(t_i,\bar x_i)\}$; cf. Figure~\ref{fig:p_prime_data} for snapshots of the approximation of $n(t,x)$ for $t\in \{2,6,9\}$. Since $n_0'(x)p'(x)\leq0$ for all $x\in (-1,1)$, all $x\in (-1,1)$ will eventually be critical points. The point $x=0$ is a critical point for all times, while each $x\neq 0$ is a critical point of $n(t,\cdot)$ exactly for one $t>0$, see Section~\ref{sec:critical}. In Figure~\ref{fig:p_prime_rec} the corresponding reconstruction $p'_r$ of $p'$ is shown.
As predicted by Theorem~\ref{thm:recon_p_d_prime}, we observe excellent agreement of the reconstruction with $p'$, which is to be expected for highly resolved approximation.

If we add $2.5\%$ of uniformly distributed noise to the location of the critical points, i.e., the data is changed to $\{(t_i,\bar x_i(1+\delta\eta))\}$ with $\eta\sim U(-1/2,1/2)$ and $\delta=0.05$, the reconstructions deteriorate, but only in a minor fashion, see Figure~\ref{fig:p_prime_rec}. In fact, employing the smoothness of the initial datum the influence of noise can be quantified by Taylor expansion. For sufficiently small noise, we obtain a linear rate of convergence in $\delta$ of the reconstruction error
\begin{align*}
	\sup_{i} |p'(\bar x_i)-p_{r}'(\bar x_i (1+\delta\eta))|,
\end{align*}
showing well-posedness of the reconstruction problem if the initial data and its derivative are available.
The saturation for small noise is due to the errors in the numerical approximation, and it can be overcome by using a finer discretization to generate the simulated data.
\begin{figure}\centering
	\includegraphics[width=.32\textwidth]{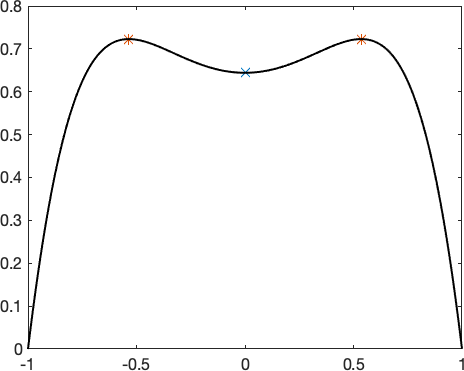}
	\includegraphics[width=.32\textwidth]{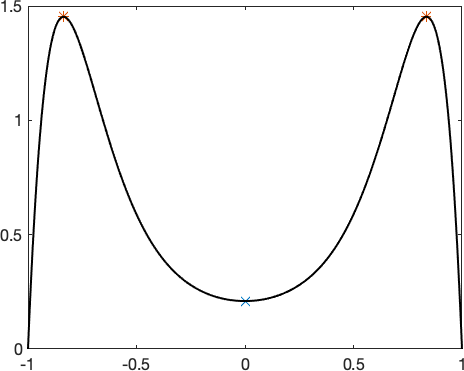}
	\includegraphics[width=.32\textwidth]{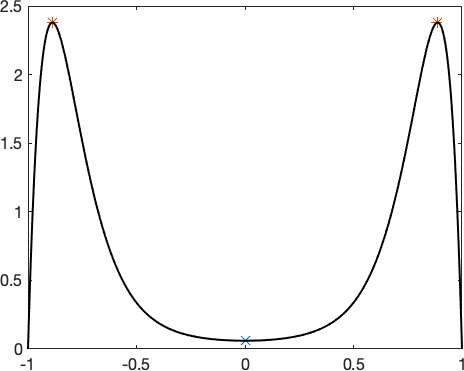}
\caption{\label{fig:p_prime_data} Snapshots of the numerical approximation of $n(t,x)$ for $t\in\{2,6,9\}$ for the reconstruction of $p'$ (from left to right). The markers denote the corresponding critical points that are used in the reconstruction formula.}
\end{figure}
\begin{figure}\centering
	\includegraphics[width=.32\textwidth]{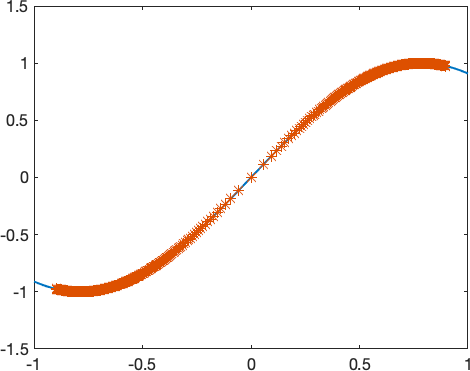}
	\includegraphics[width=.32\textwidth]{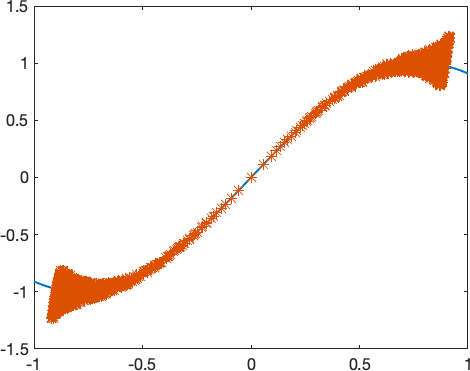}
	\includegraphics[width=.32\textwidth]{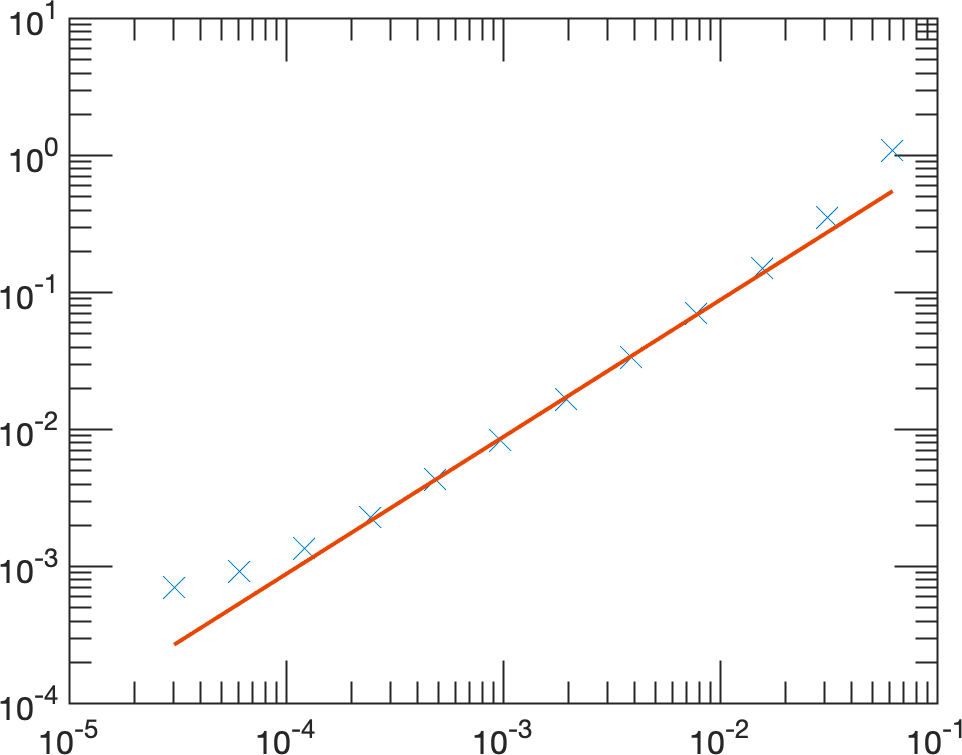}
\caption{\label{fig:p_prime_rec} Numerical reconstructions of $p'$ (red crosses) and the exact (unknown) function $p'$ (solid blue line) are shown. Left for critical points that are located within the accuracy of the numerical scheme; middle critical points with $2.5\%$ of uniform random noise. Right: Convergence rates for different noise levels $\delta=1/2^i$ for $i=4,\ldots, 15$ (crosses), the solid curve is proportional to $\delta$.}
\end{figure}
\paragraph{Reconstruction of $d'$ from critical points of $n$}
The setting is similar to the previous example. The difference is in that we choose $p(x)=1$, $d(x)=1-x^2$, and simulate until $T=3$. A similar discussion as for the previous example applies. In particular, since $d'(x)n'_0(x)>0$, all $x\in (-1,1)$ will eventually be critical points, see Section~\ref{sec:critical}. Recording the critical values of the population density and the total population allows for the reconstruction of the derivative of the unknown parameter $d$ if the initial datum is given.
Adding relative noise to the critical points will deteriorate the reconstruction only slightly; again showing well-posedness of the reconstruction problem.

\begin{figure}\centering
	\includegraphics[width=.32\textwidth]{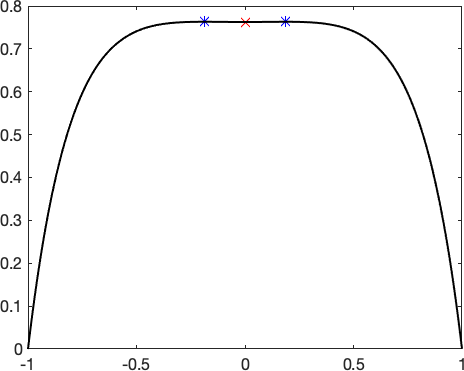}
	\includegraphics[width=.32\textwidth]{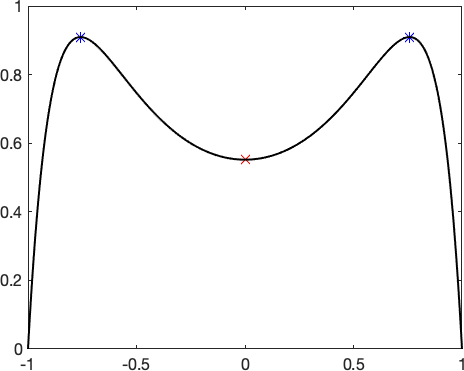}
	\includegraphics[width=.32\textwidth]{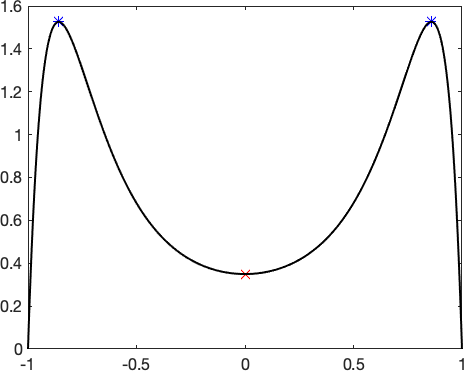}
\caption{\label{fig:d_prime_data} Snapshots of the numerical approximation of $n(t,x)$ for $t\in\{1,2,3\}$ for the reconstruction of $d'$ (from left to right). The markers denote the corresponding critical points that are used in the reconstruction formula.}
\end{figure}
\begin{figure}\centering
	\includegraphics[width=.32\textwidth]{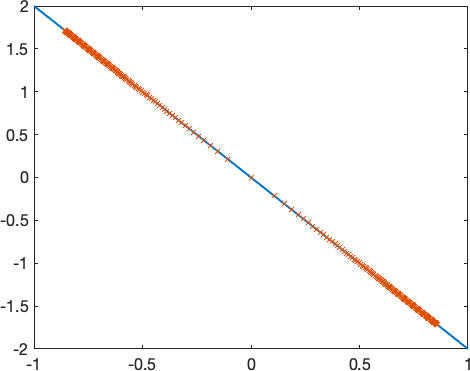}
	\includegraphics[width=.32\textwidth]{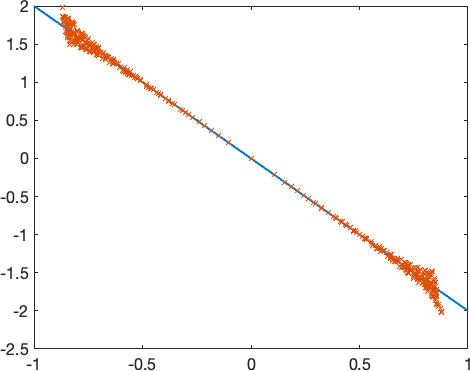}
	\includegraphics[width=.32\textwidth]{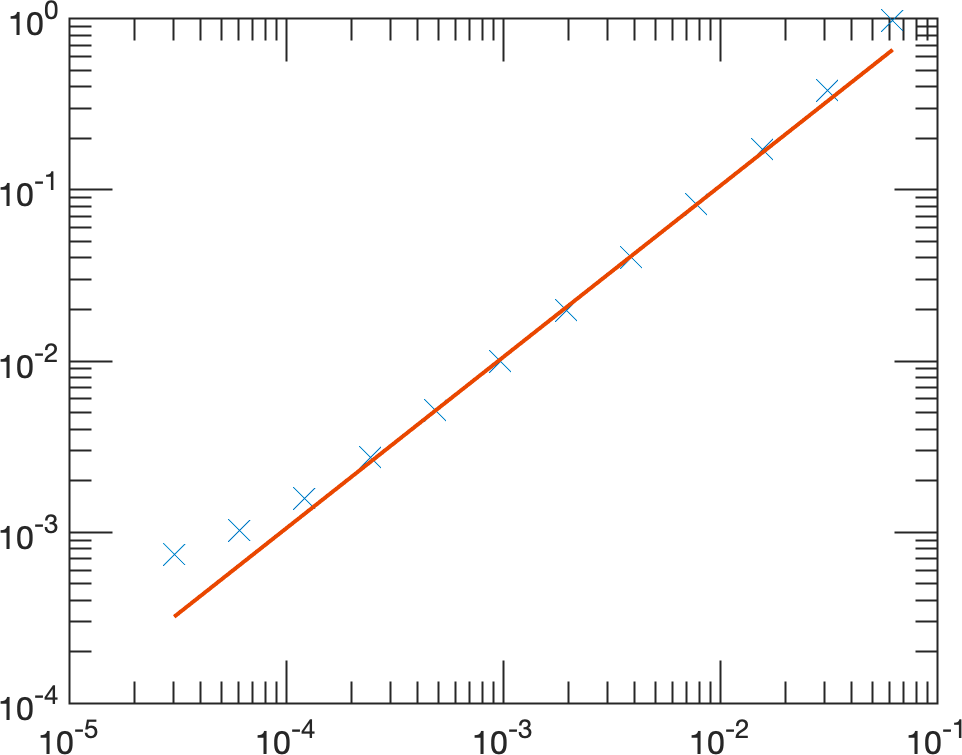}
\caption{\label{fig:d_prime_rec} Numerical reconstructions of $d'$ (red crosses) and the exact (unknown) function $d'$ (solid blue line) are shown. Left for critical points that are located within the accuracy of the numerical scheme; middle critical points with $2.5\%$ of uniform random noise. Right: Convergence rates for different noise levels $\delta=1/2^i$ for $i=4,\ldots, 15$ (crosses), the solid curve is proportional to $\delta$.}
\end{figure}

\section{Conclusions and outlook}\label{sec:outlook}
We considered several inverse problems for a nonlinear structured population model, whose dynamics is governed by a nonlocal averaging process. 
More precisely, we investigated the reconstruction of model parameters given either to total population size or the critical points of the population density. We demonstrated that in both cases the model possesses several symmetries that that leave the measurements invariant, showing the limited information content of total population size or critical points as only measurements. Ruling out these situations by appropriate assumptions on the unknown quantities, we were however able to obtain uniqueness results and in some cases explicit reconstruction formulas as well.\\
In order to model local interactions due to (small) mutations, the following generalization in the form of a parabolic system has been derived in \cite{Champagnat2006}:
 \begin{align*}
   \partial_t n(t,x) - \Delta n(t,x)  &= [p(x) - \int d(x,y)n(t,y)\;dy]n(t,x),\\ 
  n(0,x) &= n_0(x),
\end{align*}
where $d(x,y)$ allows to model more general competition behaviour. In this case, we are dealing with a second order parabolic equation and the explicit formular \eqref{eq:explicit} is no longer available. This different methods have to be applied yet we expect that some of our results can be extended to this case e.g. by using the heat kernel to obtain a fixed point equation for $\rho$. In particular, in such a setup using a perturbed forward operator as in Section~\ref{sec:perturbed} will yield a significant speed up in numerical computations. 
The investigation of such a model is, however, out of the scope of this paper and is left for future study.

\section*{Acknowledgements}
JFP acknowledges support by the German Science Foundation DFG via EXC 1003 Cells in Motion Cluster of Excellence, M{\"u}nster. The authors would like to thank Barbara Kaltenbacher (Klagenfurt) for stimulating discussions.

\bibliographystyle{plain}
\bibliography{popinv}

\end{document}